\documentclass[12pt,twoside,leqno]{amsart}
\usepackage{amsmath,mathrsfs,graphics,hyperref}
\usepackage{amssymb}
\begin{document}
\setcounter{page}{1}
\newtheorem{theorem}{Theorem}[section]
\newtheorem{lemma}[theorem]{Lemma}
\newtheorem{corollary}[theorem]{Corollary}

\vspace*{1.0cm}

\title[Operator preserving inequalities between polynomials] {On an Operator Preserving Inequalities
between Polynomials}
\author{  N. A. Rather, Suhail Gulzar   }
\date{}
\maketitle

\vspace*{-0.2cm}

\begin{center}
{\footnotesize  Department of Mathematics, University of Kashmir, Srinagar 190006, India\\
  \texttt{dr.narather@gmail.com, sgmattoo@gmail.com}

 }
\end{center}
\vskip 2mm

{\footnotesize \noindent {\bf Abstract.}   Let $\mathscr{P}_n $ denote the space of all complex polynomials $P(z)=\sum_{j=0}^{n}a_{j}{z}^{j}$ of degree   $n$ and $\mathcal{B}_n$ a family of operators that maps $\mathscr{P}_n$ into itself. In this paper, we consider a problem of investigating the dependence of
$$\left |B[P\circ\sigma](z)-\alpha B[P\circ\rho](z)+\beta\left\{\left(\frac{R+k}{k+r}\right)^{n}-|\alpha|\right\}B[P\circ\rho](z)\right| $$
on the maximum and minimum modulus of $|P(z)|$ on $|z|=k$ for arbitrary real or complex numbers $\alpha,\beta\in\mathbb{C}$ with $|\alpha|\leq 1,|\beta|\leq 1,R>r\geq k,$ $\sigma(z)=Rz,$ $\rho(z)=rz$ and establish certain sharp operator preserving inequalities between polynomials, from which a variety of interesting results follow as special cases. \vskip 1mm

\vskip 1mm

}
  \vskip 6mm

\vskip 6mm
\noindent\section*{\bf\large 1. Introduction}
\vskip 6mm

\hspace{4mm} Let $\mathscr{P}_n $ denote the space of all complex polynomials $P(z)=\sum_{j=0}^{n}a_{j}{z}^{j}$ of degree   $n$. 
A famous result known as Bernstein's inequality (for reference, see \cite[p.531]{9}, \cite[p.508]{rs} or \cite{11} states that if $P\in \mathscr{P}_n$, then
\begin{equation}\label{eq1}
\underset{\left|z\right|=1}{Max}\left|P^{\prime}(z)\right|\leq  n\underset{\left|z\right|=1}{Max}\left|P(z)\right|,
\end{equation} 
whereas concerning the maximum modulus of $P(z)$ on the circle $\left|z\right|=R>1$, we have
\begin{equation}\label{eq2}
\underset{\left|z\right|=R}{Max}\left|P(z)\right|\leq  R^{n}\underset{\left|z\right|=1}{Max}\left|P(z)\right|,\,\,\, R\geq 1.
\end{equation}
(for reference, see \cite[p.442]{8} or \cite[vol.I, p.137]{9} ).\\
 \indent If we restrict ourselves to the class of polynomials $P\in \mathscr{P}_n$ having no zero in $|z|<1$, then inequalities \eqref{eq1} and \eqref{eq2} can be respectively replaced by
\begin{equation}\label{eq3}
\underset{\left|z\right|=1}{Max}\left|P^{\prime}(z)\right|\leq        \frac{n}{2}\underset{\left|z\right|=1}{Max}\left|P(z)\right|,
\end{equation}
and
\begin{equation}\label{eq4}
\underset{\left|z\right|=R}{Max}\left|P(z)\right|\leq \frac{R^{n}+1}{2}\underset{\left|z\right|=1}{Max}\left|P(z)\right|,\,\,\, R\geq 1.
\end{equation}
Inequality \eqref{eq3} was conjectured by Erd\"{o}s and later verified by Lax \cite{7}, whereas inequality \eqref{eq4} is due to Ankey and Ravilin \cite{1}. Aziz and Dawood \cite{2} further improved inequalities \eqref{eq3} and \eqref{eq4} under the same hypothesis and proved that, 
\begin{equation}\label{eq5}
\underset{\left|z\right|=1}{Max}\left|P^{\prime}(z)\right|\leq \frac{n}{2}\left\{\underset{\left|z\right|=1}{Max}\left|P(z)\right|-\underset{\left|z\right|=1}{Min}\left|P(z)\right|\right\},
\end{equation}
\begin{equation}\label{eq6}
\underset{\left|z\right|=R}{Max}\left|P(z)\right|\leq \frac{R^{n}+1}{2}\underset{\left|z\right|=1}{Max}\left|P(z)\right|- \frac{R^{n}-1}{2}\underset{\left|z\right|=1}{Min}\left|P(z)\right|,\,\,\, R\geq 1.
\end{equation}
\indent As a compact generalization of Inequalities \eqref{eq1} and \eqref{eq2}, Aziz and Rather \cite{ar} have shown that if $ P\in\mathscr{P}_n $  then for $ \alpha,\beta\in\mathbb{C} $ with $ |\alpha|\leq 1,$ $|\beta|\leq 1,$ $ R > 1 $ and $ |z| \geq 1, $
\begin{align}\nonumber\label{e7}
  \bigg|P(Rz)-\alpha P(z)&+\beta\left\{\left(\frac{R+1}{2}\right)^{n}-|\alpha|\right\}P(z)\bigg|\\
  &\leq|z|^{n}\left|R^{n}-\alpha+\beta\left\{\left(\frac{R+1}{2}\right)^{n}-|\alpha|\right\}\right|\underset{\left|z\right|=1}{Max}\left|P(z)\right|.
  \end{align}
  The result is sharp and equality in \eqref{e7} holds for the polynomial $P(z)=az^n,$ $a\neq 0.$\\
 \indent As a corresponding compact generalization of Inequalities \eqref{eq3} and \eqref{eq4}, they \cite{ar} have also shown that if $ P\in\mathscr{P}_n $ and $ P(z) $ does not vanish in $ |z|<1,$ then for all $ \alpha,\beta\in \mathbb{C} $ with $ |\alpha|\leq 1,|\beta|\leq 1 ,$ $ R>1 $ and $|z|\geq 1,$
  \begin{align}\label{e8}\nonumber
    \Bigg|P(Rz)-\alpha P(z)+\beta&\left\{\left(\frac{R+1}{2}\right)^{n}-|\alpha|\right\}P(z)\Bigg|\\\nonumber
   \leq\frac{1}{2}&\Bigg[\left|R^{n}-\alpha +\beta\left\{\left(\frac{R+1}{2}\right)^{n}-|\alpha|\right\}\right||z|^n\\&+\left|1-\alpha +\beta\left\{\left(\frac{R+1}{2}\right)^{n}-|\alpha|\right\}\right|\Bigg]\underset{\left|z\right|=1}{Max}\left|P(z)\right|.
  \end{align}
  The result is best possible and equality in \eqref{e8} holds for $P(z)=az^n+b,$ $|a|=|b|.$\\
\indent Q. I. Rahman \cite{qir} (see also Rahman and Schmeisser \cite[p. 538]{rs}) introduced a class $\mathcal{B}_n$ of operators $B$ that carries a polynomial $P\in\mathscr{P}_n$ into 
  \begin{equation}\label{BO}
 B[P](z)=\lambda_0P(z)+\lambda_1\left(\dfrac{nz}{2}\right)\dfrac{P^{\prime}(z)}{1!}+\lambda_2\left(\dfrac{nz}{2}\right)^2\dfrac{P^{\prime\prime}(z)}{2!},
  \end{equation}\label{BO'}
  where $\lambda_0,\lambda_1$ and $\lambda_2$ are such that all the zeros of
  \begin{equation}
  U(z)=\lambda_0+n\lambda_1z+\frac{n(n-1)}{2}\lambda_2z^2
  \end{equation} 
  lie in half plane $|z|\leq\left|z-n/2\right|.$\\
\indent As a generalization of the inequalities \eqref{eq1} and \eqref{eq3}, Q. I. Rahman \cite[inequalities 5.2 and 5.3]{qir} proved that if $P\in\mathscr{P}_n,$ then
  \begin{equation}\label{qe1}
 | B[P](z)|\leq |B[z^n]|\underset{|z|=1}{Max}|P(z)|,\,\,\,\,\,\,\textnormal{for}\,\,\,\,\,\,\,|z|\geq 1,
  \end{equation}
  and if $P\in\mathscr{P}_n,$ $P(z)\neq 0$ in $|z|<1,$ then
  \begin{equation}\label{qe2}
 | B[P](z)|\leq\dfrac{1}{2} \left\{|B[z^n]|+|\lambda_0|\right\}\underset{|z|=1}{Max}|P(z)|,\,\,\,\,\,\,\textnormal{for}\,\,\,\,\,\,\,|z|\geq 1,
  \end{equation}
  where $B\in\mathcal{B}_n.$\\
In this paper, we denote for any complex functions $P,\,\rho : \mathbb C\rightarrow \mathbb C$ the composite function of  
$P$ and $\rho$, defined by
$\left(P\circ\rho\right)(z)=P\left(\rho(z)\right)\,\,(z\in\mathbb C),$  as  $P\circ\rho$.

\vskip 6mm
\section{\bf\large  Preliminaries}
\vskip 6mm

For the proof of our results, we need the following Lemmas.\\
\begin{lemma}\label{l1}
 If $P\in\mathscr{P}_n $ and $ P(z) $ have  all its zeros in $\left|z\right|\leq k$ where $k\geq 0$, then for every $R\geq r,$ $Rr\geq k^2$ and $\left|z\right|=1$, we have
\begin{equation*}
\left|P(Rz)\right|\geq \left(\frac{R+k}{r+k}\right)^{n}\left|P(rz)\right|.
\end{equation*}
\end{lemma}
The above is due to Aziz and Zargar \cite{az}. The next lemma follows from Corollary $18.3$ of \cite[p. 86]{mm}.
\begin{lemma}
If $ P\in\mathscr{P}_n $ and $P(z)$ has all zeros in $|z|\leq k,$ where $k>0$ then all the zeros of $B[P](z)$ also lie in $|z|\leq k.$ 
\end{lemma}
 \begin{lemma}\label{l3}
 If $ P\in\mathscr{P}_n $ and $ P(z) $ have no zero in $\left|z\right|<k,$ where $ k>0,$ then for all $ \alpha,\beta\in\mathbb{C} $ with $ |\alpha|\leq 1,$  $|\beta|\leq 1 $ , $ R>r\geq k $ and $ |z| \geq 1 $,
   \begin{align}\label{le3}\nonumber
       \big|B[P\circ\sigma](z)+&\Phi_k(R,r,\alpha,\beta)B[P\circ\rho](z)\big|\\
       &\leq k^n \left|B[Q\circ\tau](z)]+\Phi_k(R,r,\alpha,\beta)B[Q\circ\eta](z)\right|
       \end{align}
  where $Q(z)=z^{n}\overline{P(1/\overline{z})},$ $\sigma(z)=Rz,$ $ \rho(z)=rz,$ $\tau(z)=Rz/k^2,$ $\eta(z)=rz/k^2$ and \begin{equation}\label{phik1}
  \Phi_k(R,r,\alpha,\beta)=\beta\left\{\left(\frac{R+k}{k+r}\right)^{n}-|\alpha|\right\}-\alpha.
  \end{equation}
  \end{lemma}
  \begin{proof}
   By hypothesis, the polynomial $ P(z)$ does not vanish in $ |z|< k.$ Therefore, all the zeros of polynomial $ Q(z/k^{2}) $ lie in $ |z|<k $. As
    $$|k^{n}Q(z/k^{2})|=|P(z)| \,\,\,\, \textrm{for}\,\,\,\, |z|=k,$$
   applying Theorem \ref{t1} to $P(z)$ with $ F(z) $ replaced by $ k^{n}Q(z/k^{2}) ,$ we get for arbitrary real or complex numbers  $ \alpha,\beta $  with  $ |\alpha|\leq 1, $ $ |\beta|\leq 1 ,$  $ R>r\geq k $  and  $ |z|\geq 1, $
     \begin{align*}
     \big|B[P\circ\sigma](z)+&\Phi_k(R,r,\alpha,\beta)B[P\circ\rho](z)\big|\\
            &\leq k^n \left|B[Q\circ\tau](z)]+\Phi_k(R,r,\alpha,\beta)B[Q\circ\eta](z)\right|,
     \end{align*}
 This proves Lemma \ref{l3}.\\
  \end{proof}
   \begin{lemma}\label{l4}
  If $ P\in\mathscr{P}_n $ and $Q(z)=z^{n}\overline{P(1/\overline{z})}$  then for  $ \alpha ,\beta \in\mathbb{C} $ ,with $ |\alpha|\leq 1, |\beta|\leq 1, R>r\geq k$, $ k\leq 1 $ and $|z|\geq 1 $,
  \begin{align}\nonumber
 \Big|B[P\circ\sigma](z)+&\Phi_k(R,r,\alpha,\beta)B[P\circ\rho](z)\Big|\\\nonumber+k^n&\Big|B[Q\circ\tau](z)+\Phi_k(R,r,\alpha,\beta)B[Q\circ\eta](z)]\Big|\\&\leq \left\{|\lambda_0|\big|1+\Phi_k(R,r,\alpha,\beta)\big|   +\frac{|B[z^n]|}{k^{n}}\left|R^n+r^n\Phi_k(R,r,\alpha,\beta)\right|\right\}\underset{\left|z\right|=k}{Max}\left|P(z)\right|
  \end{align}
  where  $\sigma(z)=Rz,$ $ \rho(z)=rz,$ $\tau(z)=Rz/k^2,$ $\eta(z)=rz/k^2$ and $\Phi_k(R,r,\alpha,\beta)$ is given by \eqref{phik1}.
    \end{lemma}
    \begin{proof}
    Let $ M=Max_{\left|z\right|=k}\left|P(z)\right|, $ 
   then by Rouche's theorem, the polynomial  $ F(z)=P(z)-\mu M $ does not vanish in $ |z|<k $ for every $ \mu \in\mathbb{C} $ with $ |\mu|>1. $ Applying Lemma \ref{l3} to polynomial $ F(z) $, we get for $ \alpha, \beta \in\mathbb{C} $ with  $ |\alpha|\leq 1, |\beta|\leq 1$ and  $ |z|\geq 1 $, 
   \begin{equation*}
        \left|B[F\circ\sigma](z)+\Phi_k(R,r,\alpha,\beta)B[F\circ\rho](z)\right|
        \leq k^n \left|B[H\circ\tau](z)+\Phi_k(R,r,\alpha,\beta)B[H\circ\eta](z)\right|,
        \end{equation*}
    where  $H(z)=z^{n}\overline{F(1/\overline{z})}=Q(z)-\overline{\mu}Mz^{n}$. Replacing $ F(z) $ by $ P(z)-\mu M $ and $ H(z) $ by $ Q(z)-\overline{\mu}Mz^{n}, $ we have for $ |\alpha|\leq 1, |\beta|\leq 1$ and $ |z|\geq 1 $,
    \begin{align}\nonumber\label{l41}
 \big|B[P\circ\sigma](z)+\Phi_k(R,r,\alpha,\beta)&B[P\circ\rho](z)-\mu\lambda_0\left(1+\Phi_k(R,r,\alpha,\beta)\right)M\big|\\\nonumber
        \leq k^n \Bigg|B[Q\circ\tau](z)]&+\Phi_k(R,r,\alpha,\beta)B[Q\circ\eta](z)]\\&-\frac{\overline{\mu}}{k^{2n}}\left(R^n+r^n\Phi_k(R,r,\alpha,\beta)\right)MB[z^n]\Bigg|
    \end{align}
      where $Q(z)=z^{n}\overline{P(1/\overline{z})}$. \\
     Now choosing argument of $ \mu $ in the right hand side of inequality \eqref{l41} such that
      \begin{align*}\nonumber\label{l42}
      k^n&\bigg|B[Q\circ\tau](z)+\Phi_k(R,r,\alpha,\beta)B[Q\circ\eta](z)-\frac{\overline{\mu}}{k^{2n}}\left(R^n+r^n\Phi_k(R,r,\alpha,\beta)\right)MB[z^n]\bigg|\\\nonumber&=\frac{|\overline{\mu}|}{k^{n}}\left|R^n+r^n\Phi_k(R,r,\alpha,\beta)\right||B[z^n]|M-k^n\left|B[Q\circ\tau](z)+\Phi_k(R,r,\alpha,\beta)B[Q\circ\eta](z)]\right|
      \end{align*}
      which is possible by applying  Corollary \ref{c2}  to polynomial $ Q(z/k^{2}) $ and using the fact $ Max_{\left|z\right|=k}\left|Q(z/k^{2})\right|$ $= M/k^{n} $, we get for $ |\alpha|\leq 1, |\beta|\leq 1$  and  $ |z|\geq 1$,
      \begin{align*}\nonumber
 \big|&B[P\circ\sigma](z)+\Phi_k(R,r,\alpha,\beta)B[P\circ\rho](z)\big|-|\mu\lambda_0|\big|\left(1+\Phi_k(R,r,\alpha,\beta)\right)M\big|\\
        &\leq\frac{|\overline{\mu}|}{k^{n}}\left|R^n+r^n\Phi_k(R,r,\alpha,\beta)\right||B[z^n]|M-k^n\left|B[Q\circ\tau](z)]+\Phi_k(R,r,\alpha,\beta)B[Q\circ\eta](z)]\right|
      \end{align*}
Equivalently for $ |\alpha|\leq 1, |\beta|\leq 1$ and $ |z|\geq 1 $,
\begin{align*}\nonumber
 \big|B[P\circ\sigma](z)&+\Phi_k(R,r,\alpha,\beta)B[P\circ\rho](z)\big|+k^n\left|B[Q\circ\tau](z)]+\Phi_k(R,r,\alpha,\beta)B[Q\circ\eta](z)\right|\\&\leq|\mu| \left\{|\lambda_0|\big|1+\Phi_k(R,r,\alpha,\beta)\big|   +\frac{1}{k^{n}}\left|R^n+r^n\Phi_k(R,r,\alpha,\beta)\right||B[z^n]|\right\}M
      \end{align*}
  Letting $ |\mu|\rightarrow 1 $ ,
  we get the conclusion of Lemma \ref{l4} and this completes proof of Lemma \ref{l4}.
  \end{proof}

\vskip 6mm
\section{\bf\large  Main results}
\vskip 6mm

\begin{theorem}\label{t1}
 If $F\in \mathscr{P}_n$ and $ F(z) $  has all its zeros in the disk $\left|z\right|\leq k$ where $k>0$ and $P(z)$ is a polynomial of degree at most n such that
\begin{equation*}
\left|P(z)\right|\leq \left|F(z)\right| \,\,\, for\,\,\, |z| = k,
\end{equation*}
then for $ \left|\alpha\right|\leq 1,\left|\beta\right|\leq 1 $, $ R>r\geq k $ and  $|z|\geq 1$,  
\begin{align}\label{te1}\nonumber
\big|B[P\circ\sigma](z)+&\Phi_k(R,r,\alpha,\beta)B[P\circ\rho](z)\big|\\&
\leq \left|B[F\circ\sigma](z)+\Phi_k(R,r,\alpha,\beta)B[F\circ\rho](z)\right|
\end{align}
where $\sigma(z)=Rz,$ $ \rho(z)=rz$ and
\begin{equation}\label{phik}
\Phi_k(R,r,\alpha,\beta)=\beta\left\{\left(\frac{R+k}{k+r}\right)^{n}-|\alpha|\right\}-\alpha.
\end{equation}
The result is best possible and the equality holds for the polynomial $ P(z)= e^{i\gamma}F(z) $ where $ \gamma\in\mathbb{R} .$
\end{theorem}

 \begin{proof}[\textnormal{\textbf{Proof of Theorem \ref{t1}}}]
  Since polynomial $F(z)$ of degree $n$  has all its zeros in $|z|\leq k $ and $P(z)$ is a polynomial of degree at most $n$ such that 
   \begin{equation}\label{tp1}
   |P(z)|\leq |F(z)|\,\,\,\, \textrm{for} \,\,\,\,|z|= k,
   \end{equation}
   therefore, if $F(z)$ has a zero of multiplicity $s$ at $z=ke^{i\theta_{0}},$ $0\leq \theta_0<2\pi,$ then $P(z)$ has a zero of multiplicity at least $s$ at $z=ke^{i\theta_{0}}$. If $P(z)/F(z)$ is a constant, then inequality \eqref{te1} is obvious. We now assume that $P(z)/F(z)$ is not a constant, so that by the maximum modulus principle, it follows that\\
   \[|P(z)|<|F(z)|\,\,\,\textrm{for}\,\, |z|>k\,\,.\]
   Suppose $F(z)$ has $m$ zeros on $|z|=k$ where $0\leq m < n$, so that we can write\\
   \[F(z) = F_{1}(z)F_{2}(z)\]
   where $F_{1}(z)$ is a polynomial of degree $m$ whose all zeros lie on $|z|=k$ and $F_{2}(z)$ is a polynomial of degree exactly $n-m$ having all its zeros in $|z|<k$. This implies with the help of inequality \eqref{tp1} that\\
   \[P(z) = P_{1}(z)F_{1}(z)\]
   where $P_{1}(z)$ is a polynomial of degree at most $n-m$. Again, from inequality \eqref{tp1}, we have
   \[|P_{1}(z)| \leq |F_{2}(z)|\,\,\,for \,\, |z|=k\,\]
   where $F_{2}(z) \neq 0 \,\, for\,\, |z|=k$. Therefore
    for every real or complex number $\lambda $ with $|\lambda|>1$, a direct application of Rouche's theorem shows that the zeros of the polynomial $P_{1}(z)- \lambda F_{2}(z)$ of degree $n-m \geq 1$ lie in $|z|<k$ hence the polynomial 
    \[G(z) = F_{1}(z)\left(P_{1}(z) - \lambda F_{2}(z)\right)=P(z) - \lambda F(z)\]
    has all its zeros in $|z|\leq k$ with at least one zero in $|z|<k$, so that we can write\\
    \[G(z)= (z-te^{i\delta})H(z)\]
    where $t <k$ and $H(z)$ is a polynomial of degree $n-1$ having all its zeros in $|z|\leq k$. Applying Lemma \ref{l1} to the polynomial $H(z)$, we obtain for every $R >r\geq k $ and $0 \leq \theta <2\pi$,
   \begin{align*}
   |G(Re^{i\theta})| =&|Re^{i\theta}-te^{i\delta}||H(Re^{i\theta})|\\
    \geq& |Re^{i\theta}-te^{i\delta}|\left(\frac{R+k}{k+r}\right)^{n-1}|H(re^{i\theta})|,\\
   =&\left(\frac{R+k}{k+r}\right)^{n-1}\frac{|Re^{i\theta}-te^{i\delta}|}{|re^{i\theta}-te^{i\delta}|}|(re^{i\theta}-te^{i\delta})H(re^{i\theta})|,\\
   \geq& \left(\frac{R+k}{k+r}\right)^{n-1}\left(\frac{R+t}{r+t}\right)|G(re^{i\theta})|.
   \end{align*}
   This implies  for $R >r\geq k $ and $0 \leq \theta <2\pi$,
   \begin{equation}\label{tp2}
   \left(\frac{r+t}{R+t}\right)|G(Re^{i\theta})|\geq \left(\frac{R+k}{k+r}\right)^{n-1}|G(re^{i\theta})|.
   \end{equation}
   Since $R >r\geq k $ so that $G(Re^{i\theta})\neq 0$ for $0 \leq \theta <2\pi$ and $\frac{r+k}{k+R}>\frac{r+t}{R+t}$, from inequality \eqref{tp2}, we obtain
   \begin{equation}\label{tp3}
   |G(Re^{i\theta})|> \left(\frac{R+k}{k+r}\right)^{n}|G(re^{i\theta})|,\,\,\,\,\,\, R >r\geq k \,\,\,\, \textrm{and} \,\,\,\,\,0 \leq \theta <2\pi.
   \end{equation}
   Equivalently,
   \[|G(Rz)|> \left(\frac{R+k}{k+r}\right)^{n}|G(rz)|\]
   for $|z|=1$ and $R >r\geq k $. Hence for every real or complex number $\alpha$ with $|\alpha|\leq 1$ and $R >r\geq k,$ we have
   \begin{align}\label{tp4}
   \left|G(Rz)-\alpha G(rz)\right|&\geq\left|G(Rz)\right|-|\alpha|\left|G(rz)\right|\\\nonumber
  &>\left\{\left(\frac{R+k}{k+r}\right)^{n}-|\alpha|\right\}|G(rz)|, \,\,\,\textrm{for}\,\,\,|z|=1.
   \end{align}
   Also, inequality \eqref{tp3} can be written in the form
   \begin{equation}\label{tp5}
   |G(re^{i\theta})|<\left(\frac{k+r}{R+k}\right)^{n}|G(Re^{i\theta})|
   \end{equation}
   for every $R >r\geq k $ and $0 \leq \theta <2\pi.$ Since $G(Re^{i\theta}) \neq 0$ and $\left(\frac{k+r}{R+k}\right)^{n}<1$, from inequality \eqref{tp5}, we obtain for $0 \leq \theta <2\pi$ and $R >r\geq k$,
   \[|G(re^{i\theta})|<|G(Re^{i\theta})|.\]
   That is,
   \[|G(rz)|<|G(Rz)|\,\,\, \textrm{for}\,\,\,\, |z|=1.\]
   Since all the zeros of $G(Rz)$ lie in $|z|\leq (k/R)<1$, a direct application of Rouche's theorem shows that the polynomial $G(Rz)-\alpha G(rz)$ has all its zeros in $|z|<1$ for every real or complex number $\alpha$ with $|\alpha|\leq 1$. Applying Rouche's theorem again, it follows from \eqref{tp4} that for arbitrary real or complex numbers $\alpha,\beta$ with $|\alpha|\leq 1,|\beta|\leq 1$ and $R >r\geq k$, all the zeros of the polynomial
   \begin{align*}
   T(z)=&G(Rz)-\alpha G(rz)+\beta\left\{\left(\frac{R+k}{k+r}\right)^{n}-|\alpha|\right\}G(rz)\\
    =&\left[P(Rz)-\alpha P(rz)+\beta\left\{\left(\frac{R+k}{k+r}\right)^{n}-|\alpha|\right\}P(rz)\right]\\
    &\,\,\,\,-\lambda \left[F(Rz)-\alpha F(rz)+\beta\left\{\left(\frac{R+k}{k+r}\right)^{n}-|\alpha|\right\}F(rz)\right]
   \end{align*}
   lie in $|z|<1$.\\
   Applying Lemma \ref{l3} to the polynomial $T(z)$ and noting that $B$ is a linear operator, it follows that all the zeros of polynomial
    \begin{align*}
      B[T](z)=&\left[B[P\circ\sigma](z)-\alpha B[P\circ\rho](z)+\beta\left\{\left(\frac{R+k}{k+r}\right)^{n}-|\alpha|\right\}B[P\circ\rho](z)\right]\\
       &\,\,\,-\lambda \left[B[F\circ\sigma](z)-\alpha B[F\circ\rho](z)+\beta\left\{\left(\frac{R+k}{k+r}\right)^{n}-|\alpha|\right\}[F\circ\rho](z)\right]
      \end{align*}
      lie in $|z|<1.$ This implies
   \begin{align}\label{tp6}\nonumber
  \big|B[P\circ\sigma](z)+&\Phi_k(R,r,\alpha,\beta)B[P\circ\rho](z)\big|
  \\&\leq \left|B[P\circ\sigma](z)+\Phi_k(R,r,\alpha,\beta)B[P\circ\rho](z)\right|,
  \end{align}
   for $|z|\geq 1$ and $R >r\geq k$. If inequality \eqref{tp6} is not true, then there a point $z=z_0$ with $|z_0|\geq 1$ such that
   \begin{align*}
    \big|\big\{B[P\circ\sigma](z)+&\Phi_k(R,r,\alpha,\beta)B[P\circ\rho](z)\big\}_{z=z_0}\big|
   \\& \geq \left|\left\{B[F\circ\sigma](z)+\Phi_k(R,r,\alpha,\beta)B[F\circ\rho](z)\right\}_{z=z_0}\right|,
    \end{align*}
But all the zeros of $F(Rz)$ lie in $|z|< (k/R)<1$, therefore, it follows (as in case of $G(z)$) that all the zeros of $F(Rz)-\alpha F(rz)+\beta\left\{\left(\frac{R+k}{k+r}\right)^{n}-|\alpha|\right\}F(rz)$ lie in $\left|z\right|<1$. Hence, by Lemma \ref{l3}, 
$$\left\{B[F\circ\sigma](z)+\Phi_k(R,r,\alpha,\beta)B[F\circ\rho](z)\right\}_{z=z_0}\neq 0$$
  with $|z_0| \geq 1$.We take
     $$\lambda = \dfrac{\left\{B[P\circ\sigma](z)+\Phi_k(R,r,\alpha,\beta)B[P\circ\rho](z)\right\}_{z=z_0}}{\left\{B[P\circ\sigma](z)+\Phi_k(R,r,\alpha,\beta)B[P\circ\rho](z)\right\}_{z=z_0}},$$
     then $\lambda$ is a well defined real or complex number with $|\lambda|>1$ and with this choice of $\lambda$, we obtain $\{B[T](z)\}_{z=z_0}=0$
   where $|z_0| \geq 1$. This contradicts the fact that all the zeros of $B[T(z)]$ lie in $|z|<1$. Thus \eqref{tp6} holds for $ |\alpha|\leq 1 $, $ |\beta|\leq 1 $, $ |z|\geq 1 $, and $ R>r\geq k. $ \\
 \end{proof}
 
For $ \alpha = 0 $ in Theorem \ref{t1}, we obtain the following.

\begin{corollary}\label{c1}
 If $F\in \mathscr{P}_n$ and $ F(z) $  has all its zeros in the disk $\left|z\right|\leq k,$ where $k>0$ and $P(z)$ is a polynomial of degree at most n such that
\begin{equation*}
\left|P(z)\right|\leq \left|F(z)\right| \,\,\, for\,\,\, |z| = k,
\end{equation*}
then for $\left|\beta\right|\leq 1 $, $ R>r\geq k $ and  $|z|\geq 1$,  
\begin{align}\label{ce1}\nonumber
\bigg|B[P\circ\sigma](z)+&\beta\left(\frac{R+k}{k+r}\right)^{n}B[P\circ\rho](z)\bigg|
\\&\leq \left|B[F\circ\sigma](z)+\beta\left(\frac{R+k}{k+r}\right)^{n}B[F\circ\rho](z)\right|
\end{align}
where $\sigma(z)=Rz,$ $ \rho(z)=rz.$ 
The result is sharp, and the equality holds for the polynomial $ P(z)= e^{i\gamma}F(z) $ where $ \gamma\in\mathbb{R} .$ 
\end{corollary}
If we choose $ F(z)= z^{n}M/k^{n} $, where $ M = Max_{\left|z\right|=k}\left|P(z)\right|$ in Theorem \ref{t1}, we get the following result.
\begin{corollary}\label{c2}
If $ P\in\mathscr{P}_n $  then for $ \alpha,\beta\in\mathbb{C} $ with $ |\alpha|\leq 1,$ $|\beta|\leq 1,$ $ R >r\geq k>0$ and $ |z| = 1, $ 
  \begin{align}\label{ce2}\nonumber
\big|B[P\circ\sigma](z)+&\Phi_k(R,r,\alpha,\beta)B[P\circ\rho](z)\big|\\&
  \leq\frac{1}{k^n}\left|R^{n}+r^n\Phi_k(R,r,\alpha,\beta)\right||B[z^n]|\underset{\left|z\right|=k}{Max}\left|P(z)\right|
  \end{align}
  where $\sigma(z)=Rz,$ $ \rho(z)=rz$ and  $\Phi_k(R,r,\alpha,\beta)$ is given by \textnormal{\eqref{phik}}. The result is best possible and equality in \eqref{ce2} holds for $P(z)=az^n,$ $a\neq 0.$
\end{corollary}

Next, we take  $ P(z)= z^{n}m/k^{n} $, where $ m = Min_{\left|z\right|=k}\left|P(z)\right|$ in Theorem \ref{t1}, we get the following result.
\begin{corollary}\label{c3}
If $ F\in\mathscr{P}_n $ and $F(z)$ have all its zeros in the disk $|z|\leq k,$  where $k>0$ then for $ \alpha,\beta\in\mathbb{C} $ with $ |\alpha|\leq 1,$ $|\beta|\leq 1,$ $ R >r\geq k>0$ 
  \begin{align}\label{ce3}\nonumber
\underset{|z|=1}{Min}\big|B[F\circ\sigma](z)+&\Phi_k(R,r,\alpha,\beta)B[F\circ\rho](z)\big|\\&
  \geq\frac{|B[z^n]|}{k^n}\left|R^{n}+r^n\Phi_k(R,r,\alpha,\beta)\right|\underset{\left|z\right|=k}{Min}\left|P(z)\right|,
  \end{align}
  where $\sigma(z)=Rz,$ $ \rho(z)=rz$ and  $\Phi_k(R,r,\alpha,\beta)$ is given by \textnormal{\eqref{phik}}. The result is Sharp.
\end{corollary}

If we take $\beta=0$ in \eqref{ce2}, we get the following result.
\begin{corollary}\label{c4}
If $ P\in\mathscr{P}_n $  then for $ \alpha\in\mathbb{C} $ with $ |\alpha|\leq 1,$  $ R >r\geq k>0$ and $ |z| \geq 1, $ 
  \begin{equation}\label{ce4}
\left|B[P\circ\sigma](z)-\alpha B[P\circ\rho](z)\right|
  \leq\frac{1}{k^n}\left|R^{n}-\alpha r^n\right||B[z^n]|\underset{\left|z\right|=k}{Max}\left|P(z)\right|
  \end{equation}
  where $\sigma(z)=Rz,$ $ \rho(z)=rz.$ The result is best possible as shown by $ P(z)=az^{n}, a\neq 0. $
  \end{corollary}
\vskip 2mm

\indent For polynomials $ P\in\mathscr{P}_n $ having no zero in $ |z|< k $, we establish the following result which leads to a compact generalization of inequality \eqref{eq3},\eqref{eq4},\eqref{e8} and \eqref{qe2}.
\begin{theorem}\label{t2}
 If $ P\in \mathscr{P}_n $ and $ P(z) $ does not vanish in the disk $ |z|<k,$ where $ k\leq 1, $ then for all $ \alpha,\beta\in\mathbb{C} $ with $ |\alpha|\leq 1,$  $|\beta|\leq 1 $ , $ R>r\geq k >0$ and $ |z| \geq 1 $,
\begin{align}\label{te2}\nonumber
\big|B[P\circ\sigma](z)+&\Phi_k(R,r,\alpha,\beta)B[P\circ\rho](z)\big|
\\& \leq\frac{1}{2}\bigg[\frac{|B[z^{n}]|}{k^{^{n}}}\big|R^{n}+r^n\Phi_k(R,r,\alpha,\beta)\big|+\left|1+\Phi_k(R,r,\alpha,\beta) \right||\lambda_0|
 \bigg]\underset{\left|z\right|=k}{Max}\left|P(z)\right|
\end{align}
where $\sigma(z)=Rz,$ $ \rho(z)=rz$ and  $\Phi_k(R,r,\alpha,\beta)$ is given by \eqref{phik}.
\end{theorem}

\begin{proof}[\textnormal{\textbf{Proof of Theorem \ref{t2}}}]
  Since $ P(z) $ does not vanish in $ |z|<k,\,\, k\leq 1 $, by Lemma \ref{l3}, we have for all $ \alpha,\beta\in\mathbb{C} $ with $ |\alpha|\leq 1,\,\,|\beta|\leq 1 ,$ $ R>1 $ and $ |z|\geq 1 ,$  
  \begin{align}\label{tp'1}\nonumber
         \big|B[P\circ\sigma](z)+&\Phi_k(R,r,\alpha,\beta)B[P\circ\rho](z)\big|\\&
         \leq k^n \left|B[Q\circ\tau](z)+\Phi_k(R,r,\alpha,\beta)B[Q\circ\eta](z)\right|,
         \end{align}
     where$\sigma(z)=Rz,$ $ \rho(z)=rz,$ $\tau(z)=Rz/k^2, $ $ \eta(z)=rz/k^2$ and  $Q(z)=z^{n}\overline{P(1/\overline{z})}.$ 
   Inequality \eqref{tp'1} in conjunction with Lemma \ref{l4} gives for all $ \alpha,\beta\in\mathbb{C} $ with $ |\alpha|\leq 1,\,\,|\beta|\leq 1, $  $ R>r\geq k $ and $ |z|\geq 1 ,$ 
 \begin{align*}
       2 \big|&B[P\circ\sigma](z)+\Phi_k(R,r,\alpha,\beta)B[P\circ\rho](z)\big|\\& 
       \leq\big|B[P\circ\sigma](z)+\Phi_k(R,r,\alpha,\beta)B[P\circ\rho](z)\big|+k^n\left|B[Q\circ\tau](z)]+\Phi_k(R,r,\alpha,\beta)B[Q\circ\eta](z)\right|\\&\leq \left\{|\lambda_0|\big|1+\Phi_k(R,r,\alpha,\beta)\big|   +\frac{|B[z^n]|}{k^{n}}\left|R^n+r^n\Phi_k(R,r,\alpha,\beta)\right|\right\}|\underset{\left|z\right|=k}{Max}\left|P(z)\right|.
  \end{align*} 
   This completes the proof of Theorem \ref{t2}.\\
\end{proof} 

We finally prove the following result, which is the refinement of Theorem \ref{t2}.
\begin{theorem}\label{t3}
 If $ P\in \mathscr{P}_n $ and $ P(z) $ does not vanish in the disk $ |z|<k,$ where $ k\leq 1, $ then for all $ \alpha,\beta\in\mathbb{C} $ with $ |\alpha|\leq 1,$  $|\beta|\leq 1 $ , $ R>r\geq k>0 $ and $ |z| = 1 $,
\begin{align}\nonumber\label{te3}
\bigg|B[P\circ\sigma](z)+&\Phi_k(R,r,\alpha,\beta)B[P\circ\rho](z)\bigg|\\\nonumber
 \leq&\frac{1}{2}\Bigg[\bigg\{\frac{|B[z^{n}]|}{k^{^{n}}}\left|R^{n}+r^n\Phi_k(R,r,\alpha,\beta)\right|+\left|1+\Phi_k(R,r,\alpha,\beta) \right||\lambda_0|
 \bigg\}\underset{\left|z\right|=k}{Max}\left|P(z)\right|\\ &- \bigg\{\frac{|B[z^{n}]|}{k^{^{n}}}\left|R^{n}+r^n\Phi_k(R,r,\alpha,\beta)\right|-\left|1+\Phi_k(R,r,\alpha,\beta) \right||\lambda_0|
  \bigg\}\underset{\left|z\right|=k}{Min}\left|P(z)\right|\Bigg]
\end{align}
where $\sigma(z)=Rz,$ $ \rho(z)=rz$ and $\Phi_k(R,r,\alpha,\beta)$ is given by \eqref{phik}.
\end{theorem}

\begin{proof}[\textnormal{\textbf{Proof of Theorem \ref{t3}}}]
 Let $m=Min_{\left|z\right|=k}\left|P(z)\right|.$ If $ P(z) $ has a zero on $ |z|=k, $ then the result follows from Theorem \ref{t2}. We assume that $ P(z) $ has all its zeros in $ |z|>k $ where $ k\leq 1 $ so that $ m > 0 $. Now for every  $ \delta $ with $ |\delta|<1 $, it follows by Rouche's theorem $ h(z)=P(z)-\delta m $ does not vanish in $ |z|<k $. Applying Lemma \ref{l3} to the polynomial  $h(z),$ we get for all $ \alpha,\beta\in\mathbb{C} $ with $ |\alpha|\leq 1, |\beta|\leq 1$, $ R>r\geq k $ and $ |z|\geq 1 $
      \begin{equation*}
          \left|B[h\circ\sigma](z)+\Phi_k(R,r,\alpha,\beta)B[h\circ\rho](z)\right|
          \leq k^n \left|B[q\circ\tau](z)]+\Phi_k(R,r,\alpha,\beta)B[q\circ\eta](z)]\right|,
          \end{equation*}
    where $\sigma(z)=Rz,$ $ \rho(z)=rz,$ $\tau(z)=Rz/k^2,$ $\eta(z)=rz/k^2$ and  $q(z)=z^{n}\overline{h(1/\overline{z})}=z^{n}\overline{P(1/\overline{z})}-\overline{\delta}mz^{n}$. Equivalently,
\begin{align}\nonumber\label{tp31}
\big|B[P\circ\sigma](z)+&\Phi_k(R,r,\alpha,\beta)B[P\circ\rho](z)-\delta\lambda_0\left(1+\Phi_k(R,r,\alpha,\beta)\right)m\big|\\\nonumber
            \leq &k^n \bigg|B[Q\circ\tau](z)+\Phi_k(R,r,\alpha,\beta)B[Q\circ\eta](z)\\&\,\,\,\,-\frac{\overline{\delta}}{k^{2n}}\left(R^n+r^n\Phi_k(R,r,\alpha,\beta)\right)mB[z^n]\bigg|
        \end{align}
    where $ Q(z)=z^{n}\overline{P(1/\overline{z})}. $
 Since all the zeros of $ Q(z/k^{2}) $ lie in $ |z|\leq k,$ $ k\leq 1 $ by Corollary \ref{c3} applied to $ Q(z/k^{2}) $, we have for $ R>1 $ and $|z|=1, $ 
 \begin{align}\nonumber\label{tp32}
 \big|B[Q\circ\tau](z)]+\Phi_k(R,r,\alpha,\beta)&B[Q\circ\eta](z)]\big|\\\nonumber&\geq\frac{1}{k^{n}}\big|R^n+r^n\Phi_k(R,r,\alpha,\beta)\big||B[z^n]|\underset{|z|=k}{Min}Q(z/k^2)\\&=\frac{1}{k^{2n}}\big|R^n+r^n\Phi_k(R,r,\alpha,\beta)\big||B[z^n]|m.
 \end{align}
    Now, choosing the argument of $ \delta  $ on the right hand side of inequality \eqref{tp31}  such that
    \begin{align*}\nonumber
    k^n&\bigg|B[Q\circ\tau](z)+\Phi_k(R,r,\alpha,\beta)B[Q\circ\eta](z)-\frac{\overline{\delta}}{k^{2n}}\left(R^n+r^n\Phi_k(R,r,\alpha,\beta)\right)mB[z^n]\bigg|\\=
    &k^n\big|B[Q\circ\tau](z)+\Phi_k(R,r,\alpha,\beta)B[Q\circ\eta](z)\big|-\frac{1}{k^{n}}\big|R^n+r^n\Phi_k(R,r,\alpha,\beta)\big||B[z^n]|m.
    \end{align*}
   for $ |z|=1, $ which is possible by inequality \eqref{tp32}. We get for $ |z|=1 $ ,
   \begin{align}\nonumber\label{tp34}
\big|B[P\circ\sigma](z)+\Phi_k(R,r,\alpha,\beta)&B[P\circ\rho](z)\big|-|\delta||\lambda_0||1+\Phi_k(R,r,\alpha,\beta)\big|m\\\nonumber
\leq k^n\big|B[Q\circ\tau](z)&+\Phi_k(R,r,\alpha,\beta)B[Q\circ\eta](z
)\big|\\&\,\,\,\,-\frac{|\delta|}{k^{n}}\big|R^n+r^n\Phi_k(R,r,\alpha,\beta)\big||B[z^n]|m.
   \end{align}
Equivalently for $ |z|= 1, R>r\geq k  $, we have 
\begin{align}\nonumber\label{tp35}
\big|B[P\circ\sigma](z)&+\Phi_k(R,r,\alpha,\beta)B[P\circ\rho](z)\big|-k^n\big|B[Q\circ\tau](z)]+\Phi_k(R,r,\alpha,\beta)B[Q\circ\eta](z)]\big|\\&\leq|\delta|\left\{|\lambda_0||1+\Phi_k(R,r,\alpha,\beta)\big|
 -\frac{1}{k^{n}}\big|R^n+r^n\Phi_k(R,r,\alpha,\beta)\big||B[z^n]|\right\}m.
   \end{align}
    Letting $ |\delta|\rightarrow 1 $ in inequality \eqref{tp35}, we obtain for all $ \alpha,\beta \in\mathbb{C} $ with $ |\alpha|\leq 1, |\beta|\leq 1, R>r\geq k $ and $ |z|= 1, $
    \begin{align}\nonumber\label{tp36}
    \big|B[P\circ\sigma](z)&+\Phi_k(R,r,\alpha,\beta)B[P\circ\rho](z)\big|-k^n\big|B[Q\circ\tau](z)+\Phi_k(R,r,\alpha,\beta)B[Q\circ\eta](z)\big|\\&\leq\left\{|\lambda_0||1+\Phi_k(R,r,\alpha,\beta)\big|
     -\frac{1}{k^{n}}\big|R^n+r^n\Phi_k(R,r,\alpha,\beta)\big||B[z^n]|\right\}m.
       \end{align}
    Inequality \eqref{tp36} in conjunction with Lemma \ref{l4} gives for all $ \alpha,\beta \in\mathbb{C} $ with $ |\alpha|\leq 1,$ $ |\beta|\leq 1, R>1 $ and $ |z|=1, $
    \begin{align*}
2\big|B[P\circ\sigma](z)&+\Phi_k(R,r,\alpha,\beta)B[P\circ\rho](z)\big|\\\leq & \left\{|\lambda_0|\big|1+\Phi_k(R,r,\alpha,\beta)\big|   +\frac{1}{k^{n}}\left|R^n+r^n\Phi_k(R,r,\alpha,\beta)\right||B[z^n]|\right\}|\underset{|z|=k}{Max}|P(z)|\\&
+\left\{|\lambda_0||1+\Phi_k(R,r,\alpha,\beta)\big|
     -\frac{1}{k^{n}}\big|R^n+r^n\Phi_k(R,r,\alpha,\beta)\big||B[z^n]|\right\}\underset{|z|=k}{Min}|P(z)|.
    \end{align*}
   which is equivalent to inequality \eqref{te3} and thus completes the proof of theorem \ref{t3}.\\
\end{proof} 
If we take $\alpha=0,$ we get the following.
\begin{corollary}\label{c5}
 If $ P\in \mathscr{P}_n $ and $ P(z) $ does not vanish in $ |z|<k$ where $ k\leq 1, $ then for all $ \beta\in\mathbb{C} $ with $|\beta|\leq 1 $ , $ R>r\geq k $ and $ |z| = 1 $,
\begin{align}\nonumber\label{ce5}
  \bigg|B[P\circ\sigma](z)&+\beta\left(\frac{R+k}{k+r}\right)^{n}B[P\circ\rho](z)\bigg|\\\nonumber
 \leq&\frac{1}{2}\Bigg[\Bigg\{\frac{|B[z^n]|}{k^{n}}\left|R^{n} +r^n\beta\left(\frac{R+k}{k+1}\right)^{n}\right|
 + \left|1+\beta\left(\frac{R+k}{k+1}\right)^{n} \right||\lambda_0|\bigg\}\underset{\left|z\right|=k}{Max}\left|B[P](z)\right|\\
 & -\bigg\{ \frac{|B[z^n]|}{k^{n}}\left|R^{n}+r^n\beta\left(\frac{R+k}{k+1}\right)^{n}\right|
 -\left|1+\beta\left(\frac{R+k}{k+1}\right)^{n}\right||\lambda_0|
\bigg\}\underset{\left|z\right|=k}{Min}\left|B[P](z)\right|\Bigg]
\end{align}
where $\sigma(z)=Rz$ and $\rho(z)=rz.$
\end{corollary}
For $\beta=0,$ Theorem \ref{t2} reduces to the following result.
\begin{corollary}\label{c6}
 If $ P\in \mathscr{P}_n $ and $ P(z) $ does not vanish in $ |z|<k$ where $ k\leq 1, $ then for all $ \alpha\in\mathbb{C} $ with $ |\alpha|\leq 1,$ $ R>r\geq k $ and $ |z| = 1 $,
\begin{align}\nonumber\label{ce6}
\left|B[P\circ\sigma](z)-\alpha B[P\circ\rho](z)\right|&
 \leq\frac{1}{2}\Bigg[\left\{\frac{|B[z^n]|}{k^n}\left|R^{n}-\alpha r^n \right|
  + \left|1-\alpha \right||\lambda_0|\right\}\underset{\left|z\right|=k}{Max}\left|P(z)\right|\\&
  -\left\{\frac{|B[z^n]|}{k^n}\left|R^{n}-\alpha r^n \right|
    - \left|1-\alpha \right||\lambda_0|\right\}\underset{\left|z\right|=k}{Min}\left|P(z)\right|\Bigg]
\end{align}
where $\sigma(z)=Rz$ and $\rho(z)=rz.$ The result is sharp and extremal polynomial is $ P(z)=az^{n}+b, |a|=|b|\neq 0. $
\end{corollary}

\end{document}